\documentclass[12pt]{amsart}
\usepackage{amscd} 
\usepackage{amsfonts} 
\usepackage{amssymb} 
\usepackage{latexsym} 
\usepackage[arrow, matrix, curve]{xy}

\newcommand{\ncm}{\newcommand}

\newtheorem{theorem}{Theorem}[section]
\newtheorem{prop}[theorem]{Proposition}
\newtheorem{lemma}[theorem]{Lemma}
\newtheorem{cor}[theorem]{Corollary}
\newtheorem{lem&def}[theorem]{Lemma \& Definition}
\newtheorem{definition}[theorem]{Definition}
\newtheorem{example}[theorem]{Example}
\newtheorem{remark}[theorem]{Remark}

\def\M{\mathcal{M}}

\def\C{\mathbb{C}\,} 
\def\Z{\mathbb{Z}\,} 
 
\def\N{\mathbb{N}\,}
\def\G{\mathcal{G}}
\def\H{\mathcal{H}}
\def\One{\mathbf{1}}

\ncm{\End}{\mbox{\rm End}\,}

\def\Hom{\mbox{\rm Hom}\,}

\def\|{\, | \,}

\def\id{\mbox{\rm id}}

\def\into{\hookrightarrow}
\def\to{\rightarrow}

\def\o{\otimes}    
\def\bra{\langle}
\def\ket{\rangle}

\ncm{\rarr}[1]{\stackrel{#1}{\longrightarrow}}
\ncm{\larr}[1]{\stackrel{#1}{\longleftarrow}}
\def\cop{\Delta}

\def\eps{\varepsilon}

\def\-2{_{(-2)}}
\def\-1{_{(-1)}}
\def\0{_{(0)}}
\def\1{_{(1)}}
\def\2{_{(2)}}
\def\3{_{(3)}}
\def\n{_{(n)}}

\def\du1{\hat 1}

\begin{document}
\title[Hopf subalgebras of finite depth]{Hopf subalgebras and tensor powers of  generalized permutation modules}
\author{Lars Kadison} 
\address{Departamento de Matematica \\ Faculdade de Ci\^encias da Universidade do Porto \\ 
Rua Campo Alegre, 687 \\ 4169-007 Porto} 
\email{lkadison@fc.up.pt } 
\subjclass{16D20, 16G60, 16S34, 16T05, 19A22 }  
\keywords{subalgebra depth,  basic Hopf algebra,   normal element, smash product, finite representation type, relative Hochschild cohomology}
\date{} 

\begin{abstract}
By means of a certain  module $V$ and its tensor powers in a finite tensor category, we study a question of whether the depth of a Hopf subalgebra $R$ of a finite-dimensional Hopf algebra $H$ is finite.   The module $V$  is the counit representation induced from $R$ to $H$, which is then a generalized permutation module, as well as a module coalgebra.  We show that if in the subalgebra pair either Hopf algebra has finite representation type,  or $V$ is either semisimple with $R^*$ pointed,  projective,  or its tensor powers satisfy a Burnside ring formula over a finite set of Hopf subalgebras including $R$, then the depth of $R$ in $H$ is finite.  One assigns a nonnegative integer depth to $V$, or any other $H$-module,   by comparing the truncated tensor algebras of $V$ in a finite tensor category and so obtains upper and lower bounds for depth of a Hopf subalgebra.  For example, 
a relative Hopf restricted module has depth $1$, and a permutation module of a corefree subgroup has depth less than the number of values of its character.  
\end{abstract} 
\maketitle

\section{Introduction}

Two modules are said to be similar if each module is isomorphic to a  direct summand of a direct sum of copies of the other module: briefly formulated, each module divides a multiple of the other.  If the modules are finitely generated over a finite-dimensional algebra, similarity is equivalent by the Krull-Schmidt theorem to their having the same constituent indecomposables (i.e., isoclasses of indecomposable  direct summands).  A subalgebra $B$ in an algebra $A$ is said to have finite depth if $A^{\otimes_B \, (n+1)}$ is similar to $A^{\otimes_B n}$ for some $n \geq 0$ as $X$-$Y$-bimodules for any four choices of $X,Y \in \{ B, A \}$ \cite{BDK, LK2013}: see Section~2 for the precise definition of minimum depth  $d(B,A)$ and h-depth $d_h(B,A)$.  
It is rather easy to see that $A^{\otimes_B n}$ divides $A^{\otimes_B (n+1)}$ for all $n$, whence if $A$ is finite dimensional and any one of  $A \otimes B^{\rm op}$, 
$B \otimes A^{\rm op}$, $A^e$ or $B^e$ has finite representation type, the depth is finite, since the constituent indecomposables of the tensor powers of $A$ over $B$ are ascending subsets within a 
finite set of isoclass representative indecomposable modules.  

In \cite[Boltje-Danz-K\"ulshammer]{BDK} the depth of a finite group algebra extensions $K[H] \subseteq K[G]$, where $K$ is a commutative ring, is shown to be finite.  The authors propose the problem of determining whether the depth of a Hopf subalgebra $R$ in a finite-dimensional Hopf algebra $H$ is finite
\cite[p.\ 259]{BDK}.  In this paper we take up this problem  by showing that any tensor power of the bimodule ${}_RH_R$ is  naturally isomorphic to a smaller tensor power of $V := H/R^+H$ where $R^+$ is the augmentation ideal of $R$; more precisely, $H^{\otimes_R n} \cong H \otimes V^{\otimes n-1}$ as $H$-$H$-bimodules where the unadorned tensor is the tensor in the category $\M_H$ of finite-dimensional right $H$-modules, a finite tensor category (see \cite{EO} or Section~4 for the definition). We then define depth of a module $W$ in a finite tensor category $\mathcal{C}$ to be $n$ if the truncated tensor algebra $T_{n+1}(W)$ divides a multiple of $T_n(W)$; this condition simplifies to $V^{\otimes (n+1)} \| q V^{\otimes n}$
for a multiple $q$ if $V$ is a module coalgebra or module algebra (Prop.~\ref{prop-sepext}). It follows that  $T_{n+m}(W)$ and $T_{n+m+1}(W)$ are similar in $\mathcal{C}$ for each integer $m \geq 0$; let $d(W,\mathcal{C})$ denote the least such $n$. It is then shown that  the depth of the Hopf subalgebra satisfies $2d(V,\M_R) + 1 \leq d(R,H) \leq  2d(V,\M_R) +2$.  It follows from this that $R$ has finite depth in $H$ if $V$ is either projective, semisimple with $R^*$ pointed, $V$ is in a tensor category of finite representation type, or the left $H$-module algebra $V^*$ has a smash product with either $R$ or $H$ of finite representation type.  In Section~6,  the result in \cite{BDK} that finite group algebra extensions have finite depth is recovered in the case of an arbitrary ground field.   

\subsection{An analogy for depth of modules from number theory}
An analogy from number theory is  to ask if a sequence $\{ a_n \}_{n=1}^{\infty}$ of positive integers  is affinely generated over the primes (i.e.,  products of finitely many primes).  For example, this is the case if $a_n = a^n$ for some integer $a$, but is not the case if $a_n = p_n$, a sequence of increasing primes. As a type of toy model for what we do, say that the sequence $\{ a_n \}$ has depth $m$ if
$a_{m+1+j}$ divides a power of $a_1 \cdots a_{m+j}$ for each $j = 0,1,2,\ldots$.    Then the sequence $a_n = a^n$ has minimum depth one, and the sequence $a_n = p_n$ has infinite depth. 
As another example, the Fibonacci sequence $\{ u_n \}$ has infinite depth, since $u_{p_n}$ is a subsequence of increasing primes \cite[Theorem 179 (iv)]{HW}.  A sequence $\{a_n\}$ of minimum depth $m$
for each $m \geq 1$ may be obtained from $m$ natural numbers $u_i$ with $\mathrm{gcd}(u_1,\ldots, u_m) = 1$ by letting $a_n = u_r^{q+1}$ where $n = mq + r$ for $q \geq 0, 1\leq r \leq m$. 
The analogy is quite good for illustrating applications of the Krull-Schmidt theorem and similarity of modules with applications of unique factorization into primes and
divisibility, but can mislead if one replaces $\{ a^n \}$ with a module and its tensor powers in a finite tensor category $\mathcal{C}$ \cite{EO}, which rarely has depth $1$ (although relative Hopf restricted modules are an exception, see below); it is unclear if a module has finite depth if $\mathcal{C}$ has tame or wild representation type.

\section{Preliminaries on subalgebra depth} Let $A$ be a unital associative algebra over a field $k$. In this paper we will assume all algebras and modules to be finite-dimensional vector spaces for the sake of keeping a focus on the problem of Boltje \textit{et al}, although much below remains true without this assumption \cite{LK2013}. Two modules $M_A$ and $N_A$ are \textit{similar} (or H-equivalent) if $M \oplus * \cong qN = N \oplus \cdots \oplus N$ ($q$ times) and $N \oplus * \cong rM$
for some $r,q \in \N$.  This is briefly denoted by $M \| qN$ and $N \| rM$ $\Leftrightarrow$ $M \sim N$. 

 Let $B$ be a subalgebra of $A$ (always supposing $1_B = 1_A$).  Consider the natural bimodules ${}_AA_A$, ${}_BA_A$, ${}_AA_B$ and
${}_BA_B$ where the last is a restriction of the preceding, and so forth.  Denote the tensor powers
of ${}_BA_B$ by $A^{\otimes_B n} = A \otimes_B \cdots \otimes_B A$ for $n = 1,2,\ldots$, which is also a natural bimodule over  $B$ and $A$ in any one of four ways;     set $A^{\otimes_B 0} = B$ which is only a natural $B$-$B$-bimodule.  

If $A^{\otimes_B (n+1)}$ is similar to $A^{\otimes_B n}$ as $X$-$Y$-bimodules,  one says $B \subseteq A$ has  
\begin{itemize}
\item depth $2n+1$ if $X = B = Y$;
\item left depth $2n$ if $X = B$ and $Y = A$;
\item right depth $2n$ if $X = A$ and $Y = B$;
\item h-depth $2n-1$ if $X = A = Y$.
\end{itemize}
valid for even depth and h-depth if $n \geq 1$ and for odd depth if $n \geq 0$. 

 For example, $B \subseteq A$ has depth $1$ iff ${}_BA_B$ and ${}_BB_B$ are similar \cite{BK2, LK2013}.  In this case, it is easy to show that $A$ is algebra isomorphic to $B \otimes_{Z(B)} A^B$ where
$Z(B), A^B$ denote the center of $B$ and centralizer of $B$ in $A$.
Another example, $B \subset A$ has right depth $2$ iff ${}_AA_B$
and ${}_A A \otimes_B A_B$ are similar.  If $A = \C G$ is a group algebra of a  finite group $G$ and $B = \C H$ is a group algebra of a subgroup $H$ of $G$, then $B \subseteq A$ has right depth $2$ iff $H$ is a normal subgroup of $G$ iff $B \subseteq A$ has left depth $2$ \cite{KK}; a similar statement is true for a Hopf subalgebra $R \subseteq H$ of finite index and over any field \cite{BK}.  

Note that $A^{\otimes_B n} \| A^{\otimes_B (n+1)}$ for all $n \geq 2$
and in any of the four natural bimodule structures: one applies $1$ and multiplication to obtain a split monic, or split epi oppositely. For three of the bimodule structures, it is true for $n =1$;  as $A$-$A$-bimodules, equivalently $A \| A \otimes_B A$ as $A^e$-modules, this is the separable extension condition on $B \subseteq A$.  
But $A \otimes_B A \| qA$ as $A$-$A$-bimodules for some $q \in \N$
is the H-separability condition and implies $A$ is a separable extension of $B$ \cite{K}.  Somewhat similarly, ${}_BA_B \| q ({}_BB_B)$ implies
${}_BB_B \| {}_BA_B$ \cite{LK2013}. It follows that subalgebra depth and h-depth may be equivalently defined by replacing similarity above with $A^{\otimes_B (n+1)} \| q A^{\otimes_B n}$ for some positive integer $q$ \cite{BDK, LK2012, LK2013}.  

Note that if $B \subseteq A$ has h-depth $2n-1$, the subalgebra has (left or right) depth $2n$ by restriction of modules.  Similarly, if $B \subseteq A$ has depth $2n$, it has depth $2n+1$.  If $B \subseteq A$ has depth $2n+1$, it has depth $2n+2$ by tensoring either $-\otimes_B A$ or $A \otimes_B -$ to $A^{\otimes_B (n+1)} \sim A^{\otimes_B n}$.     Similarly, if $B \subseteq A$ has left or right depth $2n$, it has h-depth $2n+1$.  Denote the minimum depth
of $B \subseteq A$ (if it exists) by $d(B,A)$ \cite{BDK}.  Denote the minimum h-depth of $B \subseteq A$ by $d_h(B,A)$.  Note that
$d(B,A) < \infty$ if and only if $d_h(B,A) < \infty$; in fact, $| d(B,A) - d_h(B,A)| \leq 2$ if either is finite.  

For example, for the permutation groups $\Sigma_n < \Sigma_{n+1}$
and their corresponding group algebras $B \subseteq A$ over any
commutative ring $K$, one has depth $d(B,A) = 2n-1$ \cite{BKK,BDK}. Depths of subgroups in $PGL(2,q)$, twisted group algebras and Young subgroups of $\Sigma_n$ are computed in \cite{F, D, FKR}.    If $B$ and $A$ are semisimple complex algebras, the minimum odd depth is computed from powers of an order $r$ symmetric matrix with nonnegative entries $S := MM^t$ where $M$ is the inclusion matrix
$K_0(B) \rightarrow K_0(A)$ and $r$ is the number of irreducible representations of $B$ in a basic set of $K_0(B)$; the depth is $2n+1$ if $S^n$ and $S^{n+1}$ have an equal number of zero entries \cite{BKK}. (For example, the matrix $S$ has Frobenius-Perron eigenvector, the dimension vector of $B$-simples with eigenvalue $|A : B|$, the rank of the free $B$-module $A$ if $A$ and $B$ are an algebra extension of finite groups or semisimple Hopf algebras.)
Similarly, the minimum h-depth of $B\subseteq A$ is computed from
powers of an order $s$ symmetric matrix $T = M^tM$, where $s$ is the rank of $K_0(A)$, and the power $n$ at
which the number of zero entries of $T^n$ stabilizes \cite{LK2012}. 
It follows that the subalgebra pair of semisimple complex algebras $B \subseteq A$   always has finite depth.
In particular, a Hopf subalgebra of a semisimple complex Hopf algebra has finite depth, since it is semisimple \cite[2.2.2, 3.1.5]{M}.  

\subsection{Descent and going up for separable and split extensions} An algebra extension $B$ over (a subalgebra) $C$ is \textit{separable} if there is an element
$e \in B \otimes_C B$ such that $be = eb$ for all $b \in B$ and $e^1e^2 = 1_B $
in Sweedler-like notation (equivalently, the multiplication mapping $\mu: B \otimes_C B \rightarrow B$ is a split epimorphism of $B$-$B$-bimodules).  

\begin{prop}
\label{prop-separability}
Suppose $A \supseteq B \supseteq C$ is a tower of subalgebras in a finite-dimensional algebra $A$,
where $B$ is a separable extension of $C$.  If depth $d(C,A)$ or $d_h(C,A)$ is finite, then both $d(B,A)$
and $d_h(B,A)$ are finite.
\end{prop}
\begin{proof}
It will suffice to show that  $d_h(B,A) < \infty$ if $d_h(C,A) = 2n-1$.  
Then $A^{\otimes_C n} \sim A^{\otimes_C (n + s)}$ for each integer $s \geq 1$, as finite-dimensional $A^e$-modules,
which have a Krull-Schmidt theorem.  Denote the finite set of indecomposable constituents by
$\mathcal{P} := \mathrm{Indec}_{A^e}\, A^{\otimes_C  n}$.  
Let $| \mathcal{P} | = m$.  

Next note that the canonical epis $\pi_r: A^{\otimes_C r} \rightarrow A^{\otimes_B r} \rightarrow 0$
(defined by $\pi_r(a_1 \otimes_C \cdots \otimes_C a_r) = a_1 \otimes_B \cdots \otimes_B a_r$) are split as $A$-$A$-bimodule homomorphisms, since  using $e$ defined above, 
\begin{equation}
a_1 \otimes_B \cdots \otimes_B a_r \longmapsto a_1 e^1 \otimes_C e^2 a_2 e^1 \otimes_C \cdots \otimes_C e^2 a_r,
\end{equation}
is a well-defined mapping $\sigma_r: A^{\otimes_B r} \rightarrow A^{\otimes_C r}$ that satisfies
$\pi_r \circ \sigma_r = \id_{A^{\otimes_B r}}$.  It follows that $A^{\otimes_B r} \| A^{\otimes_C r}$ for each $r \geq 1$.  

Then  $A^{\otimes_B (n+s)} \| A^{\otimes_C (n + s)} \sim A^{\otimes_C n}$ for each integer $s \geq 0$, so that $\mathrm{Indec}_{A^e}(A^{\otimes_B (n+s)}) \subseteq \mathcal{P}$ for each $s \geq 0$.
Then $\mathrm{Indec}_{A^e}(A^{\otimes_B (n+m)}) = \mathrm{Indec}_{A^e}(A^{\otimes_B (n+m+1)})$, equivalently $A^{\otimes_B (n+m)} \sim A^{\otimes_B (n+m+1)}$.  Hence $d_h(B,A) \leq 2(n+m)-1$.  
\end{proof}
   
An algebra extension $A \supseteq B$ is said to be \textit{split} if there is a bimodule projection
$E: A \rightarrow B$, i.e., $E$ is $B$-$B$-bimodule mapping satisfying $E(b) = b$ for all $b \in B$.  
Separable extension and split extension have an unusual duality relationship \cite{K}.  In considering examples, it is important to note Sugano and Cohen's inequalities for the global dimensional $D(-)$ of algebras:  if
$A$ is a projective separable extension of $B$, then $D(A) \leq D(B)$, and oppositely
$D(B) \leq D(A)$ if $A$ is a projective split extension of $B$.  

\begin{prop}
\label{prop-split}
Suppose $A \supseteq B \supseteq C$ is a tower of subalgebras in a finite-dimensional algebra $A$, where $A$ is a split extension of $B$.  If $d(C,A) < \infty$, then $d(C,B) < \infty$. 
\end{prop}
\begin{proof}
Let $E: {}_BA_B \rightarrow {}_BB_B$ be a projection onto $B$.  Then the  $B$-$B$-monomorphism induced by inclusion, $0 \rightarrow B^{\otimes_C n} \rightarrow A^{\otimes_C n}$ splits for each $n \in \N$:  a splitting is given by 
$$a_1 \otimes_C \cdots \otimes_C a_n \longmapsto E(a_1) \otimes_C \cdots \otimes_C E(a_n).$$

If $d(C,A) = 2m+1$, then $A^{\otimes_C m} \sim A^{\otimes_C (m+r)}$ as $C$-$C$-bimodules, for each $r \in \N$.  It follows from $B^{\otimes_C n} \| A^{\otimes_C n}$ as $B$-$B$-bimodules established in the first paragraph, 
that $B^{\otimes_C (m+r)} \| A^{\otimes_C m}$ as $C$-$C$-bimodules
for each $r \in \N$.  As in the previous proposition and its proof, the ascending subsets
$\mathrm{Indec}_{C^e}(B^{\otimes_C (m+r)}) \subseteq \mathrm{Indec}_{C^e}(A^{\otimes_C m})$ are bounded above by a finite set for each $r \in \N$.  It follows that
$d(C,B) \leq 2(m+q) + 1$ where $q = | \mathrm{Indec}_{C^e}(A^{\otimes_C m})|$.  
\end{proof}

For a finite-dimensional algebra $A$ with subalgebra $B$, there is the Higman-Jans' theorem that $A$ has finite
representation type (f.r.t.) will imply that $B$ has f.r.t.\ if $A$ is a split extension of $B$;  conversely,
$B$ has f.r.t.\ will imply that $A$ has f.r.t.\ if $A$ is a separable extension of $B$ \cite[pp.\ 173-174, 194]{P}; in other words, f.r.t.\ descends for split extensions, and f.r.t.\ goes up for separable extensions.  

\section{Tensor powers of bimodules and modules in tensor categories}
\label{two}

Let $H$ be a finite-dimensional Hopf algebra over a field $k$ with coproduct $\cop$, counit $\eps$ and antipode $S$.   Suppose $R$ is a Hopf subalgebra of $H$.  In this case
$\cop(R) \subseteq R \otimes R$ and $S(R) \subseteq R$.  Let
$R^+$ denote $\ker \eps$, the counit restricted to $R$; e.g., $r - \eps(r)1 \in R^+$ for each $r \in R$. Note that $R^+H$ is a coideal in $H$ as well as right ideal. Form the module coalgebra (and generalized quotient of a Hopf subalgebra) $V := H/R^+H$
where $\overline{h} := h + R^+H$ and $\cop$ induces the coproduct $\cop(\overline{h}a)
 =  \overline{h\1} a\1 \otimes \overline{h\2} a\2$ for every $h,a \in H$, an $H$-module morphism, as is the counit: $V$ is a coalgebra in the module category
$\M_H$.  

For the next lemma denote the category  of Hopf subalgebras of $H$ with arrows given by inclusion by $\mathcal{H}$.  The assignment $R \mapsto V := V_{(R)}$ defines on objects a functor $F: \mathcal{H} \rightarrow \M_H$, since given another Hopf subalgebra
$R \subseteq T \subseteq H$, one has $R^+H \subseteq T^+H$ and the canonical epi of $H$-modules, $ V_{(R)} \rightarrow V_{(T)}$ $= F(R \subseteq T)$.  

Let $A$ be an arbitrary algebra. Given a (finite) bimodule category ${}_A\M_H$
and bimodule $M \in {}_A\M_H$,
there is a bifunctor $G: (M, R) \mapsto M \otimes_R H$ of ${}_A\M_H \times \mathcal{H} \rightarrow {}_A\M_H$, where $G(f: M \rightarrow N, \, R \subseteq T)$ is the canonical epimorphism $f \otimes \id_H: M \otimes_R H \rightarrow N \otimes_T H$ of $A$-$H$-bimodules.  
 The next lemma gives a natural isomorphism between the functors $G$ and $I_{{}_A\M_H} \otimes F$, where $I_{{}_A\M_H} $ is the identity functor on the category of $A$-$H$-bimodules ${}_A\M_H$. 

\begin{lemma}
\label{lem-cordoba}
Let $M$ be an $A$-$H$-bimodule.  Then $M \otimes_R H \cong M \otimes V$ 
as $A$-$H$-bimodules via
$$ \tau_{M, R}: m \otimes_R h \mapsto mh\1 \otimes \overline{h\2},$$ a natural transformation  $\tau: G \stackrel{\cdot}{\rightarrow} I_{{}_A\M_H} \otimes F$.  In particular, there is a natural isomorphism of $H$-$H$-bimodules,
$H \otimes_R H \stackrel{\cong}{\longrightarrow} {}_HH_{\cdot} \otimes V$.  
\end{lemma} 
\begin{proof}
Since $\overline{rh} = \eps(r)\overline{h}$ for every $r \in R$, this mapping is well-defined. Of course
the mapping is a left $A$-module mapping.  Recall that in the category of $H$-modules the tensor product of two
modules $U,W$ has $H$-module structure given by $(u \otimes w)h = uh\1 \otimes wh\2$ for every $u \in U, w \in W, h \in H$
(sometimes denoted by $U_{\cdot} \otimes W_{\cdot}$); the mapping in the lemma is a right $H$-module morphism where this is the $H$-module structure on $M \otimes V$.  

The mapping above has inverse mapping given by $$m \otimes \overline{h} \mapsto mS(h\1) \otimes_R h\2,$$ which is well-defined since $S(r\1 h\1) \otimes_R r\2 h\2 = S(h\1) \otimes_R \eps(r)h\2$
for all $r \in R, h \in H$.

  The naturality follows from the following commutative square of homomorphisms
of $A$-$H$-bimodules where the horizontal mappings are defined above, the vertical mappings
are induced from any choice of $f: M \rightarrow N$ in ${}_A\M_H$, the identity mapping on $H$
and $F(R \subseteq T)$:  
\begin{equation}
\label{eq: naturality}
\begin{array}{ccc}
M \otimes_R H  & \stackrel{\cong}{\longrightarrow} &  M_{\cdot} \otimes V_{(R)\cdot} \\
\downarrow      &                                                 & \downarrow \\
N \otimes_T H & \stackrel{\cong}{\longrightarrow} & N_{\cdot} \otimes V_{(T)\cdot} 
\end{array} 
\end{equation}
the diagram is commutative since $m \otimes_R h' \mapsto f(m){h'}\1 \otimes \overline{{h'}\2}$ in either
factorization.  
\end{proof}

A lemma may be similarly established for a bimodule ${}_HN_A$ where $H \otimes_R N \cong H/HR^+ \otimes N$ via a mapping $h \otimes_R n \mapsto \overline{h\1} \otimes h\2 n$.  
 
Note that 
$\dim V = \frac{\dim H}{\dim R}$ by an application of the Nichols-Zoeller freeness theorem;
in fact,  $H \cong R \otimes V$ as left $R$-modules (and right $V$-comodules, see \cite[Ch.\ 8]{M}).

\begin{lemma}
\label{lemma-right}
The right $H$-module $V = H/R^+H \cong t_RH$ where $t_R$ is any nonzero right integral in $R$.
\end{lemma}
\begin{proof}
Note the epimorphism $V \rightarrow t_R H$, $\overline{h} \mapsto t_Rh$. Let $q = \dim V$ and $h_1,\ldots,h_q \in H$ be a basis for the free module ${}_RH$.  Suppose $h = \sum_{i=1}^q r_i h_i$ and
$t_Rh = 0$, equivalently $\sum_i \eps(r_i) t_Rh_i = 0$, then $\eps(r_i)t_R = 0$, so $\eps(r_i) = 0$ for each $i = 1,\ldots, q$.  Then $h \in R^+H$.  It follows that $V \rightarrow t_RH$ is also injective.
\end{proof}

At this point, it is informative to extend \cite[Prop.\ 2.6]{LK2013}  from its hypothesis of unimodularity on $H$, with thanks to C.\ Young for the  $\iota$-method. 

\begin{cor}
\label{cor-hsepHopf}
Suppose $R\subseteq H$ is a  Hopf subalgebra of h-depth $1$.  Then $R = H$.
\end{cor}
\begin{proof}
Assume that $\iota: H \otimes V \into qH$ is an $H$-$H$-bimodule monomorphism, which must exist
since $H \otimes_R H \| qH$ for some $q \in \N$.  Let $t_H$ be a nonzero right integral in the space of integrals $\int^r_H$; let $\alpha \in H^*$ be its modular function, an augmentation of $H$ defined by $xt_H = \alpha(x)t_H$.
Then for each $v \in V$ and $h \in H$, $h\iota(t_H \otimes v) = \iota(ht_H \otimes v) = \alpha(h) \iota (t_H \otimes v)$, which implies that $\iota(t_H \otimes V) \subseteq q \int^r_H$, since each of the $q$  component is an augmented Frobenius algebra $(H,\alpha)$ with $1$-dimensional space of left integrals $k t_H$ \cite{K}; it follows that
$q \geq \frac{\dim H}{\dim R}$.  Also $\iota(t_H \otimes v)h = \iota(t_H \eps(h\1) \otimes vh\2) = \iota(t_H \otimes vh)$,
then since $\iota(t_H \otimes V) \subseteq q\int_H^r$, it follows
that $\iota(t_H \otimes vh) = \iota(t_H \otimes v\eps(h))$, whence
$vh = v\eps(h)$ for every $v \in V, h \in H$.  In particular,
$\overline{1_H}h = \overline{h} = \eps(h) \overline{1_H}$ for each $h \in H$, so that $\dim V = 1$.  Hence $\dim R = \dim H$.    
\end{proof}
  In the two extreme cases of Hopf subalgebra, when $R = H$ and $R = k1_H$, we have in the first case $V = k_{\eps}$, which 
is simple (but not projective unless $H$ is semisimple), and in the second case, $V = H$, which is free over $H$ or $R$ (and therefore not semisimple unless $H$ or $R$ is so).  

If $R$ is normal, or ad-stable in $H$, then $R^+H = HR^+$ so $V = H/HR^+$ is a trivial right
$R$-module (given by the counit).  Then taking $M = H$ in the lemma, $H \otimes_R H \cong H \otimes V \cong H^{\dim V}$ as $H$-$R$-bimodules, the right depth two condition.  The left depth two condition is similarly obtained from the variant of the lemma just mentioned; the converse that a left
(right) depth two Hopf subalgebra is right (left) ad-stable is shown in \cite{BK}.  

Note that $V \cong k \otimes_R H$, since the annihilator ideal of the $R$-module $k$ is $R^+$.  It follows that the $H$-module $V$ is the induced module of the one-dimensional trivial $R$-module;  thus, $V_H$ is relatively $R$-projective.  

\begin{example}
\label{example-subgroups}
\begin{rm}
Consider the group algebras $R = k[\mathcal{H}]$ and $H = k[\mathcal{G}]$ where $\mathcal{H} \subseteq \mathcal{G}$
is here the notation for a subgroup of a finite group.  Then $V \cong k[\mathcal{H}\setminus \mathcal{G}]$ the permutation module of right cosets (via $\overline{g} \mapsto \mathcal{H}g$). If the ground field $k = \C$,  the
character of $V$ is the induced principal character $\eta := {1_{\H}}^{\G}$.  
\end{rm}
\end{example} 
Because of the example, we might refer to the right $H$-module $V$ as the \textit{generalized permutation module} of the Hopf subalgebra pair $R \subseteq H$. 
\begin{theorem}
\label{prop-Rsemisimple}
The right $H$-module $V$ is projective if and only if $R$ is semisimple if and only if $V_R$ is projective.
\end{theorem}
\begin{proof}
($\Leftarrow$) If $R$ is semisimple, then all $R$-modules are projective.  Since $V \cong \mathrm{Ind}^H_R k$, it follows that $V_H$ is projective. (Alternatively, choose $t_R$ in the lemma above to be an idempotent.)

($\Rightarrow$) If $V$ is projective, then the short exact sequence 
\begin{equation}
\label{eq: H-module sequence}
0 \rightarrow R^+H \rightarrow H \stackrel{\pi}{\rightarrow} V \rightarrow 0
\end{equation}
 splits, where $\pi(h) = \overline{h}$; let $\sigma: V_H \rightarrow H_H$ satisfy $\pi \circ \sigma = \id_V$. Consider $e := \sigma(\overline{1})$. Then $e^2 = e$ since $\overline{h} = \overline{h'}$ iff $h-h' \in R^+H$, $\overline{e} = \overline{1}$ and 
$\sigma(\overline{1})e = \sigma(\overline{e}) = \sigma(\overline{1})$.
It follows that $V \cong eH$ via $\overline{h} \mapsto eh$.
Note that for every $r \in R$ we have $er = e\eps(r)$, since $r - \eps(r)1_H \in R^+H$.  Also note that from $e - 1 \in R^+H$ it follows that $\eps(e) = 1$.

Let $\{ h_i \}$, $\{ f_i: H_R \rightarrow R_R \}$ be dual bases for the free module $H_R$.  Then for each $i$, $f_i(e)$ is a right integral of $R$, thus $f_i(e)=c_i t_R$ for some scalar $c_i$ and a fixed nonzero right integral $t_R$ in $R$.  Then $e = \sum_i h_i f_i(e) = (\sum_i c_i h_i)t_R$, so that $\eps(t_R) \neq 0$ $\stackrel{\textrm{Maschke}}{\Rightarrow}$ $R$ is semisimple \cite[2.2.1]{M}.   

Finally, if $V_H$ is projective, so is $V_R$ since $H_R$ is free.  If $V_R$ is projective, then $V \otimes_R H$ is a projective $H$-module.  But the natural epimorphism $V \otimes_R H \rightarrow V$
is $R$-split, so $H$-split since $V$ is relatively $R$-projective.  Then $V_H$ is projective. 
\end{proof}

For example, Lorenz has shown that if $H$ is an involutory and non-semisimple Hopf algebra where $\mathrm{char}\, k = p$, then $p$ divides the dimension of any projective $H$-module \cite{Lo}:
thus if $H$ is a finite group algebra $k[\G]$, then $p | |\G|$, and if a Hopf subalgebra $k[\H]$, where $\H \leq \G$, has projective permutation $\G$-module $V$, then $p | \dim V = |\G:\H |$;  indeed consistent with the proposition since $k[\H]$ is semisimple
iff $p \not\! \| |\H|$.  

Note that the short exact sequence~(\ref{eq: H-module sequence}) may be derived from the induction functor, $-\otimes_R H$ (where ${}_RH$ is faithfully flat) applied to the $R$-module sequence:
\begin{equation}
0 \rightarrow R^+ \rightarrow R \stackrel{\eps}{\rightarrow} k \rightarrow 0 .
\end{equation}

Let $H \otimes_R \cdots \otimes_R H$ ($n$ times $H$) be denoted by $H^{\otimes_R n}$, a natural $H$-$H$-bimodule for each $n \geq 1$.  

\begin{prop}
\label{prop-tensorequation}
For any right $H$-module $W$ and each integer $n \geq 1$, there is a natural isomorphism of right $H$-modules,
\begin{equation}
\label{eq: utothen}
W \otimes_R H^{\otimes_R n} \cong W \otimes V^{\otimes n}.
\end{equation}
 In particular, 
\begin{equation}
\label{eq: fundament}
H^{\otimes_R n} \cong  H \otimes V^{\otimes (n-1)}
\end{equation}
 as $H$-$H$-bimodules, a natural isomorphism with respect to the category of Hopf subalgebras of $H$.    
\end{prop}
\begin{proof}
The statement is true for $n=1$ by Lemma~\ref{lem-cordoba} with $M = W$. Then by induction on $n >1$,
$$ W \otimes_R H^{\otimes_R n} \cong (W \otimes_R H^{\otimes_R (n-1)}) \otimes_R H \cong (W \otimes V^{\otimes (n-1)}) \otimes_R H \cong W \otimes V^{\otimes n}$$
where we apply Lemma~\ref{lem-cordoba} to  the last isomorphism in the case $M = W \otimes V^{\otimes (n-1)}$. Naturality follows from  Lemma~\ref{lem-cordoba} as well, which yields Eq.~(\ref{eq: fundament}) as an $A$-$H$-bimodule isomorphism in case $W$ is an $A$-$H$-bimodule.  Naturality also follows from a glance at the explicit formula for the isomorphism below.  
 \end{proof}
 
Either isomorphism in the proposition is given by 
\begin{equation}
\label{eq: forward}
x \otimes y \otimes \cdots \otimes z \mapsto
xy\1 \cdots z\1 \otimes \overline{y\2 \cdots z\2} \otimes \cdots \otimes \overline{z_{(n)}},
\end{equation}
with inverse mapping given by
\begin{equation}
\label{eq: inverse}
 u \otimes \overline{v} \otimes \overline{w} \otimes \cdots \longmapsto
uS(v\1) \otimes_R v\2 S(w\1) \otimes_R w\2 \cdots .
\end{equation}

Let $U_R: \M_H \rightarrow \M_H$ be the functor defined by restriction, then induction, i.e.,
$U_R(W) = W \otimes_R H$.  Then by Lemma~\ref{lem-cordoba}, $U_R$ is naturally isomorphic to the endofunctor $F_V$ of $\M_H$
defined by $W \mapsto W \otimes V$; and their iterated composites
are naturally isomorphic, ${U_R}^n(W) \cong {F_V}^n(W)$ for each $W \in \M_H$
and $n \geq 0$, where $V^{\otimes 0} := k_{\eps}$.  For example, if the ground field has characteristic zero, we think in terms of characters and 
of $U_R$ as given by $U(\chi_W) = \mathrm{Ind}^H_R\mathrm{Res}^H_R \chi_W = \chi_W\! \downarrow_R \uparrow^H$, Eq.~(\ref{eq: utothen}) implies ${U_R}^n$ is equal 
to multiplication of characters of $H$ by ${\chi_V}^n$, as expressed by Eq.~(\ref{eq: basic}) below.

\begin{example}
\begin{rm}
Suppose $R \subseteq H$ be a Hopf subalgebra pair of \textit{semisimple complex} Hopf algebras,  with $\chi$ a character of $R$ and
$\eta$ a character of $H$.  By  \cite[Lemma 5.1]{BKK}, one has 
\begin{equation}
\label{eq: Isaacs}
\chi \! \uparrow^H \! \eta \, = \, (\chi \eta \!\downarrow_R)\!\uparrow^H.
\end{equation}
Let $\chi_V$ denote the character of the generalized permutation module $V$, and $U_R$
act on any character $\eta$ of $H$ as before by $U_R(\eta) = \eta \!\downarrow_R \!\uparrow^H$: then
$\chi_V = \eps_R\! \uparrow^H$ as we noted subsequent to Example~\ref{example-subgroups}.  
Then by Eq.~(\ref{eq: Isaacs}),  $\chi_V^2 = (\eps_R (\eps_R \!\uparrow^H\! \downarrow_R)) \uparrow^H = U(\chi_V)$, 
and inductively one shows ${\chi_V}^{n+1} = {U_R}^n(\chi_V)$.  This is also implied by Eq.~(\ref{eq: utothen}), which in this
example becomes 
\begin{equation}
\label{eq: basic}
{U_R}^n(\chi_W) = \chi_W \chi_V^n
\end{equation} for all $n \geq 0$ (also noted on \cite[p.\ 151]{BKK}).
\end{rm}
\end{example}

\begin{prop}
\label{prop-sepext}
If $W$ is a finite-dimensional right $H$-module coalgebra, or dually a left $H$-module algebra,  then $W^{\otimes n} \| W^{\otimes (n+1)}$ for each $n \geq 1$
as $H$-modules (if $H$ is semisimple, $n \geq 0$).  In particular, this applies to $V$, which additionally satisfies $k_R \| V$.  
\end{prop}
\begin{proof}
Notice that the coproduct  $\cop_W: W \rightarrow W \otimes W$ is a split monomorphism of $H$-modules with respect to the counit $\eps_W$; hence, $\cop_W \otimes \id^{\otimes (n-1)}: W^{\otimes n} \rightarrow W^{\otimes (n+1)}$
is a split monic for each $n$.  Note that $\eps_W: W \rightarrow k$ is an epi of $H$-modules, which is split if $H$ is semisimple (equivalently, $k_H$ is projective).  

If $W$ is a right $H$-module coalgebra, then observe that $W^*$ is a left $H$-module algebra via
the dual algebra structure and the left $H$-module structure 
$\bra hw^*, w \ket = \bra w^*, wh \ket$ for all $w^* \in W$, $w \in W$ and $h \in H$.
It is not hard to show that $(W^*)^{\otimes n} \cong (W^{\otimes n})^*$ from which is follows from   
$W^{\otimes n} \| W^{\otimes (n+1)}$  that $(W^*)^{\otimes n} \| (W^*)^{\otimes (n+1)}$.
Alternatively, the face and degeneracy maps $A^{\otimes n} \rightarrow A^{\otimes (n \pm 1)}$ of a left $H$-module algebra $A$ are left $H$-module arrows; in particular, $A^{\otimes n} \| A^{\otimes (n+1)}$. 

Finally note that $V$ is a right $H$-module (and right $R$-module) coalgebra, and the mapping $k \rightarrow V$ given by $\lambda \mapsto
\lambda \overline{1}$ splits $\eps_V$ as an $R$-epi, since $\overline{r} = \eps(r)\overline{1}$ for each $r \in R$. 
\end{proof}
If $H$ is a right semisimple extension of $R$ (i.e. all $H$-modules are relatively $R$-projective),
then $k_H \| V_H$, since $k_H$ is relatively $R$-projective and the image of the $R$-split epi $\eps_V: V_H \rightarrow k_H$. 
The next proposition determines a projective cover and an injective hull, both within $H$, for the cyclic $H$-module $V = H/R^+H$: the standard proof (cf.\ \cite{L}) is  omitted.  
\begin{prop}
\label{prop-pc}
The projective cover of $V$ is $eH \stackrel{\pi}{\longrightarrow} V$ for some idempotent $e \in H$ with $\eps(e) = 1$ and $\ker \pi := C \subseteq e\textrm{Rad}\, H$;  there is equality of subsets if and only if $V_H$ is semisimple. There is an idempotent $f \in H$ such that the $H$-module $fH$ is the injective hull of $V$ and contains $t_RH$.
\end{prop}

\subsection{An alternative cochain complex isomorphic to the relative Hochschild complex of $(H,R)$
with coefficients} The isomorphisms in Eqs.~(\ref{eq: fundament}) and~(\ref{eq: forward}) lead naturally to an alternative cochain complex for computing relative Hochschild cohomology of a Hopf subalgebra pair $(H, R)$ with coefficients in an
$H$-$H$-bimodule $M$ \cite{Hoch}: the cochains take values in $M$ on tensor powers of the generalized
permutation module $V$, are right $H$-linear with respect to the right adjoint action on $M$
and the diagonal action on $V^{\otimes n}$, and the differentials are alternating sums of evaluation
of the counit of $V$ on successive tensorands.  The reasoning is as follows (with proofs and computations left as exercises; cf.\ \cite{Hoch}). 

 The bar resolution of $(H,R)$ is given by 
\begin{equation}
\label{eq: bar}
\cdots \stackrel{d_n}{\longrightarrow} H^{\otimes_R n} \stackrel{d_{n-1}}{\longrightarrow} H^{\otimes_R (n-1)} \longrightarrow \cdots H \otimes_R H \stackrel{d_1}{\longrightarrow} H
\end{equation}
where ($n \geq 2$, $h_1, \cdots,h_n \in H$)
$$d_{n-1}(h_1 \otimes_R \cdots \otimes h_n) = \sum_{i=1}^{n-1} (-1)^{i+1} h_1 \otimes \cdots \otimes h_i h_{i+1} \otimes \cdots \otimes h_n.$$

The bar resolution above is isomorphic via Eqs.~(\ref{eq: fundament}) and~(\ref{eq: forward}) to
\begin{equation}
\label{eq: changedcomplex}
\cdots \stackrel{{d'}_{n+1}}{\rightarrow} H_{\cdot} \otimes V_{\cdot}^{\otimes n} \stackrel{{d'}_n}{\rightarrow} H_{\cdot} \otimes V_{\cdot}^{\otimes (n-1)} \rightarrow
\cdots H_{\cdot} \otimes V_{\cdot} \stackrel{{d'}_1}{\rightarrow} H
\end{equation}
where for all $h \in H$, $v_1,\ldots,v_n \in V$, 
\begin{equation}
\label{eq: dprime}
{d'}_n (h \otimes v_1 \otimes \cdots \otimes v_n) = \sum_{i=1}^n (-1)^{i+1} \eps(v_i)
h \otimes v_1 \otimes \cdots \otimes \hat{v_i} \otimes  \cdots \otimes v_n
\end{equation}
where $\hat{v_i}$ denotes $v_i$ deleted.  

 Given a bimodule ${}_HM_H$, let $C^n(H,R;M)$ be $\Hom_{R-R}(H^{\otimes_R n}, M) \cong \Hom_{H-H}(H^{\otimes_R (n+2)}, M)$ and $\partial_n = \Hom_{H-H}(d_{n+2},M)$ with relative
Hochschild cochain complex,
\begin{equation}
\label{eq: complex}
M^R \stackrel{\partial_0}{\longrightarrow} \cdots \stackrel{\partial_{n-1}}{\longrightarrow}
\Hom_{R-R}(H^{\otimes_R n}, M) \stackrel{\partial_n}{\longrightarrow} \Hom_{R-R}(H^{\otimes_R (n+1)}, M) \rightarrow \cdots
\end{equation}
where $\partial_0(m)(h) = hm - mh$ for $m \in M^R, h \in H$ and ($f \in C^n(H,R;M)$)
\begin{eqnarray*}
(\partial_n f)(h_1, \cdots, h_{n+1}) &=&  h_1 f(h_2,\ldots,h_{n+1}) + (-1)^{n+1}f(h_1,\ldots,h_n)h_{n+1}  \\
& & + \sum_{i=1}^n (-1)^{i} f(h_1,\ldots,h_i h_{i+1}, \cdots, h_{n+1}).
\end{eqnarray*}

The relative Hochschild cohomology groups $\mathrm{H}^*(H,R;M)$ are the cohomology groups of the cochain
complex~(\ref{eq: complex}), which via the isomorphic bar resolution~(\ref{eq: changedcomplex}) and differential~(\ref{eq: dprime})
is isomorphic as cochain complexes  to $\hat{C}^n(H,R;M) :=$  $ \Hom_{-H}(V^{\otimes n}, M_{ad})$ 
($\cong \Hom_{H-H}(H_{\cdot} \otimes V^{\otimes n}, M)$ via $g \mapsto \tilde{g}$ where
$$\tilde{g}(h,v_1,\ldots,v_n) := hg(v_1,\ldots,v_n)),$$ 
and $f \in \hat{C}^n(V^{\otimes n}, M_{ad})$
 if for every $v_1,\ldots,v_n \in V$ and $h \in H$,
\begin{equation}
f(v_1 h\1, \ldots,v_n h\n) = S(h\1)f(v_1,\ldots,v_n) h\2,
\end{equation}
and the cochain differential, simplified via $g \mapsto \tilde{g}$ from
$\Hom_{H-H}({d'}_*, M)$, is given by
\begin{equation}
\label{eq: differential}
\hat{\partial}_n(f)(v_1,\ldots,v_{n+1}) = \sum_{i=1}^{n+1} (-1)^{i+1}\eps(v_i)
f(v_1,\ldots,v_{i-1},v_{i+1},\ldots,v_{n+1}).
\end{equation}
For example, $\Hom_{-H}(V,M) \cong M^R$ via $g \mapsto g(\overline{1})$, since for all $r \in R$,  $g(\overline{1})r = r\1 S(r\2) g(\overline{1}) r\3 = r\1 g(\overline{r\2}) = r g(\overline{1})$. 

If $H$ is a group algebra and $R$ is the trivial one-dimensional subalgebra, the bar resolutions
above are the homogenous and nonhomogeous bar resolutions of group cohomology theory. If
$G$ is a group, Eqs.~(\ref{eq: forward}) and~(\ref{eq: inverse}) become the chain complex  isomorphism
$$[g_1| \cdots|g_n] \longmapsto [g_1\cdots g_n| g_2\cdots g_n | \cdots | g_{n-1}g_n| g_n ]$$
with inverse $[h_1 | \cdots | h_n ] \mapsto [h_1 h_2^{-1} | h_2 h_3^{-1} | \cdots | h_n ]$
for all $g_i, h_i \in G$.

\section{A module's depth in a tensor category}
\label{three}
  
We define the depth of an object $W$ in a finite tensor category $(\mathcal{C}, \otimes, \One)$ that is equivalent as tensor categories to a category $\M_H$ of finite-dimensional right modules
over a finite-dimensional Hopf algebra $H$ over an arbitrary field $k$.  
Let $W^{\otimes n} = W \otimes \cdots \otimes W$ ($n$ times) where
$W^{\otimes 0}$ denotes the unit module $\One = k_{\eps}$. 
(Following \cite[2.1]{EO}, a tensor category is an abelian category with a tensor functor that is distributive over direct sums, satisfying rigidity conditions for its left and right duals, and has a simple unit module $\One$.  A tensor category $\mathcal{D}$ is a \textit{finite tensor category} over an algebraically closed field $\tilde{k}$ if every object has finite length, $\mathcal{D}$ has finitely many simple objects, and each simple $X$ has projective cover $P(X)$.  This is equivalent to $\mathcal{D} \cong \M_A$ for some finite-dimensional algebra $A$
over $\tilde{k}$. If $\mathcal{C}$ has a fiber functor $F: \mathcal{D} \rightarrow \tilde{k} \mathrm{-Vect}$,  
a Tannakian reconstruction shows that $\mathcal{D} \cong \M_A$ where $A$ is a finite-dimensional Hopf algebra. The hypothesis that the ground field be algebraically closed is important in \cite{EO}, but not necessary for what we do below.)   
Proposition~\ref{prop-tensorequation} suggests that defining a depth of $W$ in $\M_R$ is useful for obtaining an upper bound on depth
$d(R,H)$ of a Hopf subalgebra $R \subseteq H$.  We make use of the truncated tensor algebra of $W$,  $T_n(W) := W \oplus (W \otimes W) \oplus \cdots \oplus W^{\otimes n}$, since   $T_n(W)$ clearly divides $T_{n+1}(W)$ in the abelian category $\mathcal{C}$ for each $n \geq 1$; for $n =0$ define
$T_0(W) = \One$.  (There is no harm done including 
$\One$ in the definition of $T_n(W)$ if dealing only with  fusion categories and semisimple Hopf algebras. See the discussion about semisimple extensions after Proposition~\ref{prop-sepext}. ) 

\begin{definition}
\label{def-moduledepth}
\begin{rm}
Say that the nonzero object $W$  has depth $n \geq 0$ in $\mathcal{C}$ if $T_{n+1}(W)$, equivalently $W^{\otimes (n+1)}$, divides a multiple of $T_n(W)$; briefly, $W^{\otimes (n+1)} \| qT_n(W)$. Equivalently  
$W$ has depth $n$ iff 
$T_n(W) \sim T_{n+1}(W)$ in $\mathcal{C}$.  Note 
that  if $W$  has depth $n$ in $\mathcal{C}$, then it has depth $n+1$, which follows
from tensoring $T_{n+1}(W) \| qT_n(W)$ by $W$ and adding $W$'s to both sides
to obtain $T_{n+2}(W) \| qT_{n+1}(W)$. Denote the minimum depth of $W$, if it exists, by $d(W, \mathcal{C})$.   Write $d(W, \mathcal{C}) = \infty$ if $W$ has no finite depth. 
\end{rm}
\end{definition}

Note that for a coalgebra or algebra $W$ in the tensor category $\mathcal{C}$ (i.e., $W$ is a right $H$-module algebra or coalgebra) we may simplify the definition of depth $n$ object to the equivalent conditions $W^{\otimes (n+1)} \| qW^{\otimes n}$, or
$W^{\otimes n} \sim W^{\otimes (n+1)}$, by making use of Proposition~\ref{prop-sepext}.  

\begin{example}
\begin{rm}
For a module $W$ over a group algebra, the (finite depth) condition $T_{n+1}(W) \| qT_n(W)$ for some $q,n \in \N$ 
is the condition for $W$ being an \textit{algebraic module}. 
Let $\H \subseteq \G$ be a subgroup of a finite group, and $k = \C$; the
character of the permutation module $V$ is the induced principal character $\eta := {1_{\H}}^{\G}$ as noted above in Example~\ref{example-subgroups}.  This character is faithful if $\H$
is corefree in $\G$, i.e., there are conjugates of $\H$ with trivial intersection \cite{I}.  In this case, the Brauer-Burnside theorem \cite[p.\ 49]{I}  asserts that each irreducible character $\psi$ of $\G$ is a constituent of a power $\eta^n$ where $0 \leq n < m \leq |\G:\H|$ and $m$ is the number of distinct values taken by $\eta(g)$ as $g \in \G$.  Putting Prop.~\ref{prop-sepext} together with Def.~\ref{def-moduledepth}, one obtains that $\eta^{m-1}$ contains each irreducible character as a constituent, so the depth $d(V, \M_{\C\! [\G]}) \leq m-1$.  (Looking ahead and applying Theorem~\ref{theorem-main}, we see that as a consequence the depth of the corefree subgroup is $d_{\C}(\H, \G) \leq 2m \leq 2|\G : \H |$. ) 
\end{rm}
\end{example}
Next are a series of lemmas  for computing depth of modules in finite tensor categories;
 we switch back to our point-of-view in the module category
$\M_H$ of a finite-dimensional Hopf algebra $H$.   The following   lemma may be applied to either inclusions of Hopf subalgebras or epis to Hopf algebra quotients. 

\begin{lemma}
\label{lemma-RES}
Given a Hopf algebra homomorphism $f: R \rightarrow H$,  if depth $d(W, \M_H) < \infty$, then depth of its restriction along $f$ satisfies $$d(W_f, \M_R) \leq d(W, \M_H).$$  
\end{lemma}
\begin{proof}
We first note that the functor $U \mapsto U_f$ is a tensor functor from $\M_H \rightarrow \M_R$, since $(U \otimes W)_f = U_f \otimes W_f$ as $R$-modules follows from the coalgebra morphism property of $f$, and $\mathbf{1}_f = \mathrm{1}_R$ from $\eps_H \circ f = \eps_R$. If $W^{\otimes (n+1)} \| T_n(W)$ for some integer $n \geq 0$, then ${W_f}^{\otimes (n+1)} \| T_n(W_f)$ follows readily.    
\end{proof}
 
  For example, if $R$ is an ad-stable Hopf subalgebra of $H$,  then $V =  H/R^+H = \overline{H}$, the quotient Hopf algebra, which is a Hopf module: these have depth $d(V, \M_{\overline{H}}) \leq 1$ as noted below in Corollary~\ref{cor-one}. Then by Lemma~\ref{lemma-RES} applied to $H \rightarrow \overline{H}$,  $d(V,\M_H) \leq 1$, then applied to $R \into H$,  $d(V, \M_R) \leq 1$. On the other hand, $R^+H = HR^+$, so that $V =  H/HR^+$ is a trivial right $R$-module with depth $0$.
\begin{lemma}
Given a module $W \in \M_H$, let $\mathcal{P}_n(W)$ denote the set of isomorphism classes of indecomposable direct summands of $T_n(W)$.  
\begin{itemize}
\item Then $\mathcal{P}_n(W) \subseteq \mathcal{P}_{n+1}(W)$, with equality iff $W$ has depth $n$ in $\M_H$;
\item If $U \| W$ in $\M_H$, then $d(W,\M_H) \leq n$ implies $d(U, \M_H) \leq n + |\mathcal{P}_n(W)|$;
\item If $H$ is a quasi-triangular Hopf algebra and $V$, $W$ are right $H$-module coalgebras of finite depth in $\M_H$, then  $V \otimes W$ is a right $H$-module coalgebra of depth, 
\begin{equation}
\label{eq: braided}
d(V \otimes W, \M_H) \leq \max \{ d(V, \M_H), d(W, \M_H) \}
\end{equation}
\end{itemize}   
\end{lemma}
\begin{proof}
The inclusion follows from Krull-Schmidt applied to $T_n(W) \oplus W^{\otimes (n+1)} \cong T_{n+1}(W)$.
The opposite inclusion $\mathcal{P}_n(W) \supseteq \mathcal{P}_{n+1}(W)$ follows in the same way from $T_{n+1}(W) \oplus * \cong qT_n(W)$ (indeed, equivalently).

If $U \| W$, then using distributive law, one obtains $U^{\otimes n} \| W^{\otimes n}$ for each $n \in \N$.   It follows that $\mathcal{P}_n(U) \subseteq \mathcal{P}_n(W) = \mathcal{P}_{n+k}(W)$
for all $k \geq 0$.  Then $\mathcal{P}_{n+s}(U) = \mathcal{P}_{n+s+1}(U)$ for $s = |\mathcal{P}_n(W)|$.   

Let $R = R^1 \otimes R^2$ be an R-matrix for $H$, so that $R$ is invertible in \newline $H \otimes H$ and $R^1h\1 \otimes R^2h\2 = h\2 R^1 \otimes h\1 R^2$ for all $h \in H$.  It is well-known that $\M_H$ is braided, so that 
$V \otimes W \stackrel{\cong}{\longrightarrow} W \otimes V$ via $v \otimes w \mapsto
 w R^1 \otimes v R^2$ as arrows in $\M_H$; cf. \cite[10.1.2]{M}.   Now it is not difficult to check that, if $V$ and $W$ are right $H$-module coalgebras, then $V \otimes W$ is itself a module coalgebra
with coproduct using the braiding as follows, $\cop(v \otimes w) = v\1 \otimes w\1 R^1 \otimes v\2 R^2 \otimes w\2$;  since $(\cop_H \otimes \id_H)(R) = R^{13}R^{23}$, $(\id_H \otimes \cop_H)(R) = R^{13}R^{12}$, coassociativity
follows, and since $\eps_H(R^1)R^2 = 1 = R^1 \eps_H(R^2)$, the counit defined by 
$\eps(v \otimes w) = \eps_V(v)\eps_W(w)$ satisfies the counit equations.

By braided commutativity then, $(V \otimes W)^{\otimes m} \cong V^{\otimes m} \otimes W^{\otimes m}$ in $\M_H$ for each $m \in \N$.  If $V$ and $W$ are 
finite depth $H$-module coalgebras, then there 
are $q_1, q_2 \in \N$ and complementary $H$-modules $U_1, U_2$ such that $V^{\otimes (n+1) } \oplus U_1 \cong q_1 V^{\otimes n}$
and $W^{\otimes (n+1)} \oplus U_2 \cong q_2 W^{\otimes n}$.  Tensoring
and deleting summands with some $U_i$ as a factor yields
$$(V \otimes W)^{\otimes (n+1)} \cong V^{\otimes (n+1)} \otimes W^{\otimes (n+1)} \| q_1q_2 V^{\otimes n} \otimes W^{\o n} \cong q_1q_2 (V \otimes W)^{\o n}.$$
The inequality~(\ref{eq: braided}) follows from this. 
\end{proof}

Suppose $s = | \{$ isomorphism classes of simples in $\M_H \} | = | \{$ isomorphism classes of projective indecomposables in $\M_H \} |  = | \{$ isomorphism classes of injective indecomposables in $\M_H \} |$ (via bijection $X \mapsto P(X)$) \cite{DK}.  

\begin{prop}
If $W$ is a projective module in $\M_H$, then its depth satisfies $d(W, \M_H) \leq s$.  
\end{prop}
\begin{proof}
The nonzero tensor powers of $W$ are projective modules in $\M_H$  (\cite[Prop.\ 2.1]{EO}, also noted after Proposition~\ref{prop-withconsequence} for $k$ not
necessarily algebraically closed). 
Therefore the indecomposable summands of $W^{\otimes n}$ are projective indecomposables of which there are $s$ different modules.    Since  $\mathcal{P}_n(W) \subseteq \mathcal{P}_{n+1}(W)$ for all $n > 0$, and these are bounded above by a finite set of cardinality $s$, it follows that 
$\mathcal{P}_s(W) = \mathcal{P}_{s+1}(W)$, so $W$ has depth $s$ in $\M_H$.  
\end{proof}

A Hopf algebra $H$ has the Chevalley property, i.e., the tensor product of two simple $H$-modules is semisimple iff the radical ideal of $H$ is a Hopf ideal \cite{Lo}.  If this is not the case, we must distinguish a proper subset, the maximal nilpotent Hopf ideal $J_{\omega}(H) \subset \mathrm{rad}\, H$ \cite{CH}. 
 As a note of caution for the next proposition,  consider the special case where $k$ has characteristic $p$ and $H$ is a group algebra $k[G]$.   It is noted in \cite{CH} that $J_{\omega} (H) = R^+H$ for the Hopf subalgebra $R = k[O_p(G)]$, where $O_p(G)$ is the largest normal $p$-subgroup of $G$, based on the result in \cite{PQ} that  (nilpotent) Hopf ideals in $H$ are of the form $Hk[N]^+$ for normal ($p$-) subgroups $N \leq G$.
Thus for the Hopf subalgebra pair $R = k[O_p(G)] \subseteq k[G] = H$, the permutation module $V \cong H/J_{\omega}(H)$ is isomorphic to the Hopf algebra $k[G/O_p(G)]$, which is a semisimple $H$-module iff $O_p(G)$ is the Sylow $p$-subgroup iff $k[G]$ has the Chevalley property \cite[Theorem 4.4]{CH}.  

\begin{prop}
\label{prop-semisimple}
 Suppose the radical ideal $J$ of a Hopf algebra $H$ is a Hopf ideal (e.g., $H^*$ is a pointed Hopf algebra \cite[5.2.8]{M}).  If $W$ is a semisimple $H$-module, then $d(W, \M_H) \leq s$.
\end{prop}
\begin{proof}
If the radical is just a coideal, it is in fact a Hopf ideal, in which case the set of semisimple $A$-modules is closed under tensor product \cite[cor.\ 8]{PQ}.  Then $W$ and its powers are semisimple modules made up of $s$ different simples.  The rest of the proof proceeds as the proof of the previous proposition. 
\end{proof}

The paper \cite{PQ} proves that if an $H$-module $W$ has annihilator ideal that contains no nonzero Hopf ideal of $H$, then the tensor $H$-module algebra $T(W)$ is faithful as an $H$-module.  
The paper \cite{LL} classifies pointed Hopf algebras $H$ of finite corepresentation type over an algebraically closed field; equivalently, classifies basic Hopf algebras $H^*$ of finite representation type.   The paper \cite{AEG} classifies triangular Hopf algebras over $\C$ with the Chevalley property.  

\begin{remark}
\begin{rm}
The generalized permutation module $V = H / R^+H$ is a semisimple $R$-module if any one the following three conditions is met: 
1) if $R$ is an ad-stable Hopf subalgebra (as noted above);  
2)  if the radical ideal $J$ of $R$ is left ad-stable in $H$ (for a short computation shows that $HJ \subseteq JH$, then $VJ = 0$);
3) if $\mathrm{rad}\, R \subseteq \mathrm{rad}\, H$ and $V_H$ semisimple.
\end{rm}
\end{remark}

Finally suppose  the finite tensor category $\M_H$ has only $t$ isoclasses of indecomposables for some $t \in \N$; i.e., the Hopf algebra $H$ has finite representation type. In this case we have the proposition below by an argument similar to the proofs in the propositions directly above.

\begin{prop}
\label{prop-frt}
If $H$ has finite representation type, then a module $W$ has depth $d(W, \M_H) \leq t$.  
\end{prop}

Recall that over an algebraically closed field $\overline{k}$, an algebra $A$ is \textit{basic} if it has only one-dimensional irreducible representations; over $k$, if $A / J$ is a product of division algebras, where $J$ denotes the maximal nilpotent ideal 
$\mathrm{rad}\, A$.  
By \cite[3.1]{LL}, a basic Hopf algebra over $\overline{k}$  has finite representation type (f.r.t.) if and only if it is a Nakayama algebra (i.e. each projective indecomposable has a unique composition series \cite{ASS}).  

\begin{example}
\label{ex-8dim}
\begin{rm}
Let $q$ be a primitive $n$'th root of unity in $k = \C$, and $d = n$ if $n$ is odd, 
$d = \frac{n}{2}$ if $n$ is even.  Note that $q^2$ is a primitive $d$'th root of unity. 
Let $ H_q =  \overline{U_q}(sl_2(\C))$ be  the algebra generated by $K,E,F$ where  $K^d = 1$, $E^d = 0 = F^d$, $EF - FE = \frac{K  - K^{-1}}{q - q^{-1}}$, $KE = q^2EK$, and $KF = q^{-2}FK$.  This is a $d^3$-dimensional Hopf algebra with coproduct given by $\cop(K) = K \otimes K$, $\cop(E) = E \otimes 1 + K \otimes E$ and $\cop(F) = F \otimes K^{-1} + 1 \otimes F$ (\cite{R} for an alternative set of generators and relations).  

The Hopf algebras $H_q = \overline{U_q}(sl_2)$ are basic and do not have finite representation type  for a direct reason worth noting here.  According to \cite[Theorem 3.2]{ASS}, the ordinary quiver of $H_q$ is that of  a Nakayama algebra only if at each vertex there is at most one outgoing and one incoming arrow.  Next identify the  orthogonal primitive idempotents as $e_{i+1} =  \sum_{j=0}^{d-1}(q^{2i}K)^j/d$
for $i = 0,\ldots,d-1$; note that they satisfy $e_{i+1}E = Ee_{i+2} \in e_{i+1}(J/J^2)e_{i+2}$, $Fe_{i+1} = e_{i+2}F \in e_{i+2}(J/J^2)e_{i+1}$ (where $e_{d+1} = e_1$), which form a cycle of $d$ vertices, each with two incoming and two outgoing arrows. (See \cite{JK} for when the wider class of  Frobenius-Lusztig kernels and their blocks have tame or wild representation type.)  For example, the following diagram represents the quiver  for $d = 4$: 

$$\begin{array}{ccc}
\bullet^1 & \rightleftharpoons  & \bullet^2 \\
\uparrow \downarrow &  & \uparrow \downarrow  \\
\bullet^4 &  \rightleftharpoons & \bullet^3  
\end{array}$$

We consider in more detail $H = H_i =  \overline{U_q}(sl_2(\C))$ at the $4$'th root of unity $q = i$, which is an $8$ dimensional algebra  generated  by $K,E,F$ where $K^2 = 1$, $E^2 = 0 = F^2$, $EF = FE$, $KE = -EK$, and $KF = -FK$.  Let $R$ be the Hopf subalgebra of dimension $4$ generated by $K,F$ (isomorphic to the Taft algebra considered below in Example~\ref{ex-taft}).  Note that neither $R$ nor $J := \mathrm{rad}\, R$ are ad-stable, and that $H$ is a basic algebra, self-dual and pointed.  We denote the two simples pulled back along $H \rightarrow H/J \cong \C^2$ by $S_1, S_2$. 

Consider
$V := H / R^+H$ as an $R$-module, which is spanned by $\overline{1}$ and $\overline{E}$, where $\overline{1}K = \overline{1}$, $\overline{E}K = -\overline{E}$, and $F$ annihilates.  
From Theorem~\ref{prop-Rsemisimple}, one notes that $V_R$ is not projective, which can also be obtained from Doi's observation:  the module coalgebra $V$ is projective iff
there is a right $R$-module map $\psi: C \rightarrow R$ such that $\eps_V \psi = \eps_R$ \cite{D}.  

The projective cover is $e_1H$ where $e_1 = \frac{1+K}{2}$ and $\pi: e_1H \rightarrow V$ has kernel
$C$ spanned by $e_1F$ and $e_1EF$.  
Note that $V_H$ is not semisimple, since $C \neq e_1\textrm{Rad}\, H$.
However, $V_R$ is a semisimple module, since the radical ideal $J$ in $R$ (spanned by $F$ and $FK$ in $R^+$) satisfies $HJ = JH$.    
The injective hull of $V \cong t_RH$  where $t_R = F(1 + K)$
is $e_2H = (\frac{1-K}{2})H$.

 The composition series of $e_1H$ and $e_2H$ have length $4$ and are non-unique, with radical length  $3$:  $e_1H \supset e_1J \supset e_1J^2 \supset \{ 0 \}$ (a second and third proof that $H$ is not Nakayama \cite{ASS}). The Cartan map $K_0(H) \rightarrow G_0(H)$ is given by $[xH]\mapsto 2[S_1] + 2[S_2]$ for both
$x = e_1,e_2$,  a symmetric matrix; indeed (cf.\ \cite{L}) $H$ (and any $\overline{U_q}(sl_2)$)
has $S^2$ an inner automorphism and is unimodular, whence it is a symmetric algebra.

\end{rm}
\end{example}
\subsection{Relative Hopf modules} Next we show that finite-dimensional relative Hopf restricted modules have depth $1$.  Again let $R \subseteq H$ be a Hopf subalgebra pair.  Recall that a vector space $N$ is a right $(H,R)$-Hopf module if $N$ is a right $R$-module, and $N$ is a right $H$-comodule such that $(nr)\0 \otimes (nr)\1 = n\0 r\1 \otimes n\1 r\2$
for all $r \in R, n \in N$; if $N$ is moreover the restriction of an $H$-module, we refer to it as a relative Hopf restricted module.  
The next proposition extends \cite[Lemma 3.1.4]{M} for our purposes.  

\begin{prop} 
\label{prop-withconsequence}
Given a relative Hopf module N and right $H$-module $M$, the following is an isomorphism of right $R$-modules: \begin{equation}
N_. \otimes M \cong N_. \otimes M_.
\end{equation}
where the right-hand side is the tensor product in $\M_R$.
\end{prop}
\begin{proof}
The forward mapping is given by $n \otimes m \mapsto n\0 \otimes m n\1$.  The inverse mapping is given by $n' \otimes m' \mapsto {n'}\0 \otimes m' \overline{S}({n'}\1)$, where $\overline{S}$ is the composition-inverse of the antipode.  
\end{proof}  
Since each free module $H^n$ over a Hopf algebra $H$ is a Hopf module (and relative Hopf module with respect to any Hopf subalgebra), it follows from  Proposition~\ref{prop-withconsequence} that $P \otimes M$
is projective in $\M_H$ for any projective $P \in \M_H$ for any ground field $k$.  

\begin{cor}
\label{cor-one}
Suppose $N$ is a right $(H,R)$-Hopf restricted module. Then depth $d(N, \M_R) \leq 1$.  
\end{cor}
\begin{proof}
Let $M = N$ in the proposition, so that $N \otimes N \cong (\dim N)N$ in additive notation, a depth $1$ condition in $\textrm{Mod-}R$.  
\end{proof} 
If one asks when the module $V$, formed from a Hopf subalgebra $R \subseteq H$, is a relative Hopf restricted module, a necessary condition for this is that $V_R$ is free, since relative Hopf modules are free by
a result of Nichols-Zoeller \cite{M}.  So a necessary condition
for $V$ to be relative Hopf module is that $\dim H = (\dim R)^2 r$ for some integer $r \geq 1$, and that   $R$ is semisimple by Proposition~\ref{prop-Rsemisimple}.  
If $V$ is a relative Hopf module, it follows from Theorem~\ref{theorem-main} below that the depth $d(R,H) < 5$.  The next proposition provides a sufficient condition for $V$ to be a relative Hopf module; the proof is straightforward and left to the reader. 
  
\begin{prop}
Given $R \subseteq H$ a finite Hopf subalgebra pair, suppose there is a module morphism $\phi: V_R \rightarrow H_R$ of modules that is simultaneously a coalgebra morphism.  Then $\rho := (V \otimes \phi) \circ \cop_V$ defines a right $H$-comodule structure on $V$ making $V$ an $(H,R)$-Hopf module.
\end{prop}
It follows from the injectivity of $\rho$ and $\cop_V$ that $\phi$ must be injective.   

\subsection{Smash products and Hopf subalgebras}
It is interesting to pursue the depth of the left $H$- or $R$-module algebra $V^*$ where $V = H /R^+H$ is a right $H$- and $R$-module coalgebra, Note that there is  an augmentation $ \alpha: V^* \to k$ defined by $v^* \mapsto v^*(\overline{1_H})$: this augmentation  extends to an augmentation  $\alpha \otimes \eps: V^* \# R \to k$ for the smash product algebra of $V^*$ with $R$ \cite[Ch.\ 4]{M}.   
The module depths are equal, $d(V, \M_H) = d(V^*, {}_H\M)$, since $(V^{\otimes n})^* \cong (V^*)^{\otimes n}$ and $V^{**} \cong V$.   Note that the dual of the epi $H \stackrel{\pi}{\rightarrow} V$ of right $H$-module coalgebras is the monomorphism of $H$-module algebras
$\pi^*: V^* \rightarrow H^*$: the mapping $E: H^* \rightarrow V^*$, defined using Lemma~\ref{lemma-right} by $h^* \mapsto h^*(t_R -)$, is an arrow in the category ${}_{V^*}\M_{V^*}$, which is a Frobenius homomorphism with dual bases $\{ S^{-1}(\lambda\2) \}$, $\{ \lambda\1 \}$, where $\lambda$ is a right integral in $R^*$ such that $\lambda(t_R) = 1$;  this data makes $H^*$ a \textit{Frobenius extension} of $V^*$
\cite{K}.
The mapping $E: {}_HH^* \rightarrow {}_HV^*$ is also left $H$-linear.  Further, $H^*$ is a split extension of $V^*$  if  $R$ is semisimple.
The mapping $\pi^*$  induces the monomorphism on smash products $V^* \# H \into H^* \# H \cong M_{\dim H}(k)$, embedding $V^* \# H$ in the Heisenberg double of $H$ \cite[Ch.\ 9]{M}.  

We record
some of these observations in the proposition below, the details of the proof being left to the reader. Thinking of the self-dual noncocommutative Taft Hopf algebra $H$ in Example~\ref{ex-taft} and $R$ the cyclic group subalgebra in $H$, one passes first to the cocommutative $H$-module coalgebra $V$, then
to the commutative Frobenius-Nakayama subalgebra $V^*$ generated by the nilpotent element $x$, which is a well-known ordinary (i.e., non-twisted) split Frobenius extension $V^* \into H$ \cite{K}. 
The proposition offers a generally applicable method of passing from the twisted Frobenius extension $R \subseteq H$ of a Hopf subalgebra
to an ordinary Frobenius extension $V^* \into H^*$.

\begin{prop}
Given a finite Hopf subalgebra pair $R \subseteq H$, there is $ H \rightarrow V \rightarrow 0$ the canonical epimorphism of right $H$-module coalgebras. Then the dual monomorphism of left $H$-module algebras $$0 \longrightarrow V^* \longrightarrow H^*$$ is an ordinary free Frobenius extension of rank $\dim R$ (with Frobenius coordinate system given above).  
\end{prop}

\begin{remark}
\label{remark-smash}
\begin{rm}
It is shown in \cite{CY} that the subalgebra depth of a Hopf algebra $H$
in a smash product with any left $H$-module algebra $A$ satisfies 
\begin{equation}
\label{eq: young}
d_{\mathrm{odd}}(H, A \# H) = 2 d(A, {}_H\M) + 1,
\end{equation}
where $d_{\mathrm{odd}}(R,H) = 2\lceil \frac{d(R,H)-1}{2} \rceil +1$ is the least odd integer greater than or equal to $d(R,H)$, based on cancellations of the type $H \otimes_H M_i \cong M_i$ for  left $H$-modules
$M_i$ in the tensor power $H$-$H$-bimodule $(A \# H)^{\otimes_H n}$.  As we show in Corollary~\ref{cor-equivalence}, this implies that the Hopf subalgebra $R \subseteq H$ has finite depth if and only if depth of $H$ in the smash product $V^* \# H$ is finite.
\end{rm}  
\end{remark}


\section{Equalities between Hopf subalgebra depth and quotient module depth}
\label{sec: UB}
Again suppose $R$ is a Hopf subalgebra of a Hopf algebra $H$ with  $V$ denoting the right $R$-module $H/ R^+H$.  Let $d_{\mathrm{even}}(R,H) = 2 \lceil \frac{d(R,H)}{2} \rceil$ be the least even number greater than or equal to $d(R,H)$.  
\begin{theorem}
\label{theorem-main}
The depths of a Hopf subalgebra and its generalized permutation module $V$ are related by  $d_{\mathrm{even}}(R,H) = 2d(V,\M_R)+2$. Moreover the h-depth satisfies $d_h(R,H) = 2 d(V,\M_H) +1$.  
\end{theorem}
\begin{proof}
   Suppose $d(V,\M_R) = n$.  Then $V^{\otimes n} \sim V^{\otimes (n+1)}$ as right $R$-modules, since
$V$ is a right $R$-module coalgebra. It follows from tensoring by ${}_HH_. \otimes -$ and applying Prop.\ \ref{prop-tensorequation} twice that $H^{\otimes_R (n+1)} \sim H^{\otimes_R (n+2)}$ as $H$-$R$-bimodules;
this is the right depth $2n+2$ condition. It follows that 
$d_{\mathrm{even}}(R,H) \leq 2d(V,\M_R) +2$.

Conversely, if $R$ in $H$ has depth $2n+2$, then $H^{\otimes_R (n+1)} \sim H^{\otimes_R (n+2)}$ as $H$-$R$-bimodules; consequently, ${}_HH_{\cdot} \otimes V_{\cdot}^{\otimes n} \sim {}_HH_{\cdot} \otimes V_{\cdot}^{\otimes (n+1)}$ as $H$-$R$-bimodules from two applications of Prop.\ \ref{prop-tensorequation}. Tensoring both sides of the similarity by $k \otimes_H -$, one obtains cancellation and $V^{\otimes n} \sim V^{\otimes (n+1)}$ as right $R$-modules.  It follows that $d(V,\M_R) \leq n$, so that $d_{\mathrm{even}}(R,H) = 
2n+2 \geq 2 d(V, \M_R) +2$.  The two inequalities above yield the equality in the first statement of the theorem.  

 Suppose $d(V, \M_H) = n \geq 1$.  Then arguing as in the first paragraph of the proof, one arrives at  $H^{\otimes_R (n+1)} \sim  H^{\otimes_R (n+2)}$ as  $H$-$H$-bimodules, which is the h-depth $2n+1$ condition.  
Then $d_h(R,H) \leq 2d(V, \M_H) +1$.   Conversely, if $d_h(R,H) = 2n+1 \geq 3$, we arrive similarly to the second paragraph at ${}_HH_{\cdot} \otimes {V_{\cdot}}^{\otimes n} \sim {}_HH_{\cdot} \otimes {V_{\cdot}}^{\otimes (n+1)}$, obtaining cancellation by applying the additive
functor $k \otimes_H -$ to both sides of the similarity: whence $V^{\otimes n} \sim V^{\otimes (n+1)}$ as right $H$-modules.  It follows that $d_h(R,H) \geq 2d(V, \M_H) +1$ and that the equality in the second statement of the theorem holds.  

Finally, if $d(V, \M_H) = 0$, then $V \| k \oplus \cdots \oplus k$, and so ${}_HH_{\cdot} \otimes V_{\cdot} \| {}_H(H \oplus \cdots \oplus H)_H$.  It follows from Prop.~\ref{prop-tensorequation}
that ${}_HH \otimes_R H_H$ divides a multiple of $H$, the h-depth $1$ condition (H-separability). 
Conversely, if the h-depth $1$ condition is satisfied by $R \subseteq H$,  Cor.~\ref{cor-hsepHopf} shows that $R= H$, whence $V = H/H^+ \cong k$ has depth $0$. 
\end{proof}

\begin{example}
\begin{rm}
Suppose $H$, hence $R$, are semisimple complex Hopf algebras with $r \times s$ inclusion matrix $M$. As outlined in the introduction to depth in Section~2, the h-depth $d_h(R,H)$
is determined by the $s \times s$ matrix $T = M^tM$ which has $i,j$-entries $\bra U_R(\chi_i) , \chi_j \ket_H$  $ = $  $ \bra \chi_i \downarrow_R, \chi_j \downarrow_R \ket_R$ where $\chi_i, \chi_j \in \mathrm{Irr}\, (H)$ ($1 \leq i,j \leq s$).  The h-depth is $2n+1$ iff the matrix $T^n$ has equally
many zero entries as $T^{n+1}$.  In this case we may read the result of the theorem above in this limited setting from
Eq.~(\ref{eq: basic}), which equivalently states that for each $i,j = 1,\ldots,s$, 
\begin{equation}
\bra U^n(\chi_i),\, \chi_j \ket = \bra \chi_i {\chi_V}^n, \, \chi_j \ket
\end{equation}
for each $n \geq 0$. 
The increasing powers of the $H$-character $\chi_V$ stop acquiring new irreducible constituents,
$\mathrm{Irr}\, ({\chi_V}^n) = \mathrm{Irr}\, ({\chi_V}^{n+1}) \subseteq \mathrm{Irr}\, (H)$ at $n = d(V, \M_H)$.  
\end{rm}
\end{example}
It follows from the implication h-depth $2n-1 \Rightarrow$ depth $2n$ and from the theorem that the depths of $R \subseteq H$ and the $H$-module $V$ are related by
\begin{equation}
\label{eq: main}
 2d(V, \M_H) - 1 \leq d(R,H) \leq 2d(V, \M_H) + 2,
\end{equation}
whereas the depths of $R \subseteq H$ and $V_R$ are more closely related as
\begin{equation}
2d(V, \M_R) +1 \leq d(R,H) \leq 2d(V, \M_R) + 2.
\end{equation}

Recall that an element $x \in H$ is normal if $xH = Hx$. Let $t_R$ be a right integral in $R$.  The following proposition may be viewed as a type of generalization to nonsemisimple Hopf algebras of Masuoka's theorem that $t_R$ is central in a semisimple Hopf algebra $H$ iff $R$ is ad-stable in $H$. 

\begin{prop}
A Hopf subalgebra $R$ is ad-stable in $H$ if and only if $t_R$ is a normal element in $H$.  
\end{prop}
\begin{proof}
If $t_R$ satisfies $t_RH = Ht_R$, then $V$ is a trivial $R$-module by Lemma~\ref{lemma-right}. Then $V$ has depth $0$ in $\M_R$; so by Theorem~\ref{theorem-main}, $d(R,H) \leq 2$.  Then $R$ is ad-stable in $H$ by the characterization of depth two Hopf subalgebras in \cite{BK}.

If $R$ is ad-stable in $H$, then for each $h \in H$, $h\1 t_R S(h\2) \in R$.  Furthermore, 
$V$ is a trivial right $R$-module, from which it follows from $V \cong t_RH$ as right $H$-modules
(Lemma~\ref{lemma-right}) that $t_Rhr = t_Rh\eps(r)$ for all $h \in H, r \in R$.  
Then $h\1 t_R S(h\2) r = h\1 t_R S(h\2) \eps(r)$ for each $h \in H, r \in R$; whence $h\1 t_R S(h\2) = \mu t_R$ for some scalar $\mu$. Then
$$ ht_R = h\1 t_R S(h\2) h\3 = \sum_i \mu_i t_R h_i  \in t_RH$$
for some $h_i \in H$ and $\mu_i \in k$.   Then $Ht_R \subseteq t_RH$ and similarly we argue
$t_RH \subseteq Ht_R$.  
\end{proof}

As an example of applications of Theorem~\ref{theorem-main},  if $\mathrm{rad}\, R$ is ad-stable in  $H$ with $R^*$ pointed, 
then $d(R,H) < \infty$ by  Proposition~\ref{prop-semisimple}, since $V$ and its tensor powers are semisimple $R$-modules. Secondly, the computations of depth
of subgroups in \cite{BKK, BDK, F, FKR} computes depth of the permutation modules over the subgroup,
while computations of h-depth of Hopf subalgebras in \cite{LK2012} computes depth of their generalized permutation modules with respect to the overalgebra:  e.g., if $R = H_8$, the $8$-dimensional semisimple Hopf algebra in
\cite[Example 3.5]{LK2012} and \cite[p.\ 528]{R}, and $H = D(H_8)$ its Drinfeld double, the depths
$d(R,H) = 3$ and $d_h(R,H) = 5$,  whence the generalized permutation module has depths $d(V, \M_R) = 1$ and $d(V, \M_H) = 2$.

If we consider the complex group algebra extension of the permutation groups naturally embedded, $R = \C S_n$ and $H = \C S_{n+1}$, the ordinary depth $d(\C S_n,\C S_{n+1}) = 2n-1$, while $d_h(\C S_n, \C S_{n+1}) = 2n+1$ as noted in the lemma below; whence the depths of the associated permutation module $V$ are 
\begin{equation}
 d(V, \M_{S_n}) = n-1, \ \ d(V, \M_{S_{n+1}}) = n.
\end{equation} 

\begin{lemma}
The h-depth $d_h(\C S_n, \C S_{n+1}) = 2n+1$. 
\end{lemma}
\begin{proof}
Since $d(\C S_n, \C S_{n+1}) = 2n-1$ \cite{BKK},  $d_h(\C S_n, \C S_{n+1}) \leq 2n+1$ follows. 
The h-depth is the minimum odd depth of the transpose inclusion matrix $M^t$ \cite{LK2012}.  The minimum odd depth of any inclusion matrix $M$ (of nonnegative integer entries which we may assume is connected) is $1 +$ the diameter of the black dots in the associated bipartite graph, where there is
a black dot for each row, a white dot for each column and these are connected by an edge if the corresponding entry in $M$ is positive \cite{BKK}.   The odd depth of $M^t$ is then the diameter of white dots:
since $M$ is an inclusion matrix, the white dots represent irreducible representations of $\C S_{n+1}$.
The lemma will follow by exhibiting the Young diagrams (labelled by partitions of $n+1$) of two white dots $i,j$ with distance $d(i,j)$ in edges \cite{D}
equal to $2n+1$. The branching rule and the argument in \cite{BKK} that $d((n), 1^n) = 2n-2$ shows that
the distance $d((n+1), 1^{n+1}) = 2n$.  A sketch of the two additional edges in the shortest path from
$(n+1)$ to $1^{n+1}$, necessarily via the shortest path from $(n)$ to $1^n$,  is given below:

$$\begin{array}{cccccccccccc}
(n+1) &                &           &               & (n,1) &                &              &  (2,1^{n-1})    &         &      & &1^{n+1} \\
         & \searrow &           & \nearrow &           & \searrow  &            &\vdots       &  \searrow &  & \nearrow & \\
         &               & (n)      &               &            &              & (n-1,1) &               &         &1^n & & 
\end{array} $$
\end{proof}
\begin{cor}
\label{cor-equivalence}
With the notation above,
\begin{equation}
d(R,H) - d(R, V^* \# R) \leq 2
\end{equation}
and 
\begin{equation}
d_h(R,H) = d_{\mathrm{odd}}(H, V^* \# H)
\end{equation}
\end{cor}
\begin{proof}
Note that $V$ restricts to a right $R$-module coalgebra, then from the proof of Proposition~\ref{prop-sepext} we recall that
$V^*$ is a left $R$-module algebra; thus the smash product $V^* \# R$ is defined
such that $R$ is naturally embedded as a subalgebra (see \cite[ch.\ 4]{M}).  Also $d(V, \M_R) =$
$ d(V^*, {}_R\M)$, so that
$d_{\mathrm{even}}(R,H) = 2 d(V^*, {}_R\M) +2$ follows from Theorem~\ref{theorem-main};
from Eq.~(\ref{eq: young}) in Remark~\ref{remark-smash} one deduces that
$$d_{\mathrm{odd}}(R,V^* \# R) = 2d(V^*, {}_R\M) + 1 = d_{\mathrm{even}}(R,H) - 1.$$ Then $d(R,H) = d(R, V^* \# R)$ if both are odd integers,
$d(R,H) - d(R,V^*\# R) = 2$ if both are even, and their difference is $1$ if the two depths have different parity. 

The equality of depths follows from the second statement in Theorem~\ref{theorem-main}, the
fact that $d(V, \M_H) = d(V^*, {}_H\M)$ and Eq.~(\ref{eq: young}) applied to the left $H$-module algebra  $A = V^*$. 
\end{proof}
\begin{example}
\label{ex-taft}
\begin{rm}
Consider the Taft Hopf algebras $H = H_n = \bra g,x \| \, g^n = 1, x^n = 0, gxg^{-1} = qx \ket$
where $q$ is a primitive $n$'th root of unity in the ground field $k$.  These are Hopf algebras, where
$g$, $x$, is a grouplike, skew primitive element, \cite{R} of dimension $n^2$  (for each $n$ the Taft algebras are examples of basic algebras, Nakayama algebras, pointed Hopf algebras, algebras having finite representation and corepresentation type,  as well as monomial algebras \cite{LL}).  Consider the cyclic group subalgebra $R \cong k\Z_n$ affinely generated by $g$ in $H$.  The category $\M_R$
is semisimple with $n$ simples; thus the theorem computes the depth $d(R,H) \leq 2n+2$ by
an application of Prop.\ \ref{prop-semisimple}.  We improve this by computing $V$
and its depth.  Since $V = H/R^+H$ is $n$-dimensional, and $R^+$ is spanned by $1 - g^j$ for $j = 1,\ldots, n-1$, $V$ is spanned by $\overline{x^i}$ where $i = 0,1,\ldots,n-1$.  The $R$-module action on $V$ is given by $\overline{1}g^i = \overline{1}$, $\overline{x^j}g^i = q^{-ij}\overline{x^j}$.  (Thus
$V$ is a free $R$-module via $V_R \rightarrow R_R$, $\overline{x^{n-i}} \mapsto e_i$ where
$e_i = \frac{1}{n} \sum_{j=0}^{n-1}(q^{-i}g)^j$.  As an $H$-module $V \cong e_0 H$  is not semisimple, and as a coalgebra $V$ is cocommutative.) We compute that $V \otimes V \cong nV$, 
since $$\overline{x^j} \otimes \overline{x^k} \mapsto (0,\ldots,\overline{x^{[j+k]}},0,\ldots,0)$$ where the nonzero term occurs in the $k+1$'st coordinate and $[j+k]$ is congruent to $j+k$ (mod $n$)
and in the interval $0 \leq [j+k] < n$:  this follows from noting $(\overline{x^j} \otimes \overline{x^k})g^i = q^{-i(j+k)} \overline{x^j} \otimes \overline{x^k}$.
For example, if $n=2$ 
\begin{eqnarray*}
\overline{1} \otimes \overline{1} \mapsto (\overline{1}, 0) & & \overline{1} \otimes \overline{x} \mapsto (0,\overline{x}) \\
\overline{x} \otimes \overline{1} \mapsto (\overline{x},0) & & \overline{x} \otimes \overline{x} \mapsto (0,\overline{1}) 
\end{eqnarray*}
 It follows that
 $d(V,\M_R) = 1$.  Then by Theorem, $3 \leq d(R,H) \leq 4$.
It is computed by other means in \cite{CY} that $d(R,H) = 3$. 

\end{rm}
\end{example}
\begin{example}
\begin{rm}
The $4$- and $8$-dimensional Hopf subalgebra pair in Example~\ref{ex-8dim} has $d(R,H) \leq 6$
since $d(V, \M_R) \leq 2$.  More generally, the Taft Hopf algebra $R_d = H_d \subseteq \overline{U_q}(sl_2) = H_q$ is a Hopf subalgebra
of the $d^3$-dimensional quantum group at the $n$'th root of unity $q$, defined in the same example.
My computations at the $n$'th root of unity for $d > 2$ show that $V_{R_d}$ is not semisimple: a basis for $V$ is given by $\{ \overline{E^j}: j = 0,1,\ldots,d-1 \} $ and $\overline{E^2}\cdot F = -(q+q^{-1})\overline{E}$, whence $VJ \neq 0$. However, the radical $J$ of $R_d$ satisfies $J^d = 0$; by the formula for Nakayama algebras, \cite[p. 170]{ASS}, there are at most $d^2$ nonisomorphic indecomposable $R_d$-modules. It follows that the depth $d(R_d, H_q) \leq 2d^2+2$.  
\end{rm}
\end{example}
Again let $R \subseteq H$ denote a Hopf subalgebra. The ground field $k$ is of arbitrary characteristic,
and we do not require $R$ or $H$ to be basic.  (Basic Hopf algebras of finite representation type
are characterized in \cite{LL} as being Nakayama algebras.)
\begin{cor}
The  depth $d(R,H)$ is finite if any one of the following conditions are met:
\begin{enumerate}
\item either $R$ or $H$ has finite representation type;
\item the generalized permutation module $V$ is projective;
\item the $R$-module $V$ is semisimple and $R$ has the Chevalley property
(e.g. $R^*$ is pointed).
\end{enumerate} 
\end{cor}
\begin{proof}
  If $R$ has f.r.t.\ then the right $R$-module $V = H/R^+H$ has finite depth by  Proposition~\ref{prop-frt}.  Then $d(R,H) < \infty$ by Theorem~\ref{theorem-main}.  If $H$ has f.r.t.\ with $n$ isomorphism  classes of indecomposable modules, then the right $H$-module $V$ has  depth $ n$, and its restriction along $R \into H$ has depth $d(V, \M_R) \leq n$ by Lemma~\ref{lemma-RES}.  Thus  $d(R,H) \leq 2n + 2$ by Theorem~\ref{theorem-main}.  

If $V = H/ R^+H$ is projective as either an $R$- or $H$-module, then by Proposition~\ref{prop-Rsemisimple}, $R$ is a  semisimple algebra, and therefore has f.r.t. Finally, if $V$ is a semisimple $R$-module,
and semisimple modules are closed under tensor product, then $V$ has finite depth as noted in
Proposition~\ref{prop-semisimple} and the paragraph above it.  Then $d(R,H) < \infty$ follows from Theorem~\ref{theorem-main}.  
\end{proof}

\begin{remark}
\begin{rm}
In order to deal with the depth of a finite-dimensional Hopf subalgebra pair $R \subseteq H$, both of tame or wild representation type, it will not work to find a split extension of the overalgebra to a Hopf algebra having f.r.t.\ and use Prop.~\ref{prop-split} to conclude that $d(R,H) < \infty$ ; nor will it work to find  the Hopf subalgebra separably extending down to a Hopf subalgebra having f.r.t., then using Prop.~\ref{prop-separability} to conclude that depth is finite.  The reason is that the Higman-Jans theorem, stated in the same subsection, contradicts both of  these approaches. 
\end{rm}
\end{remark}

\subsection{Depth of tensor Hopf algebras} 
If $R_i \subseteq H_i$ are Hopf subalgebras for each $i = 1,\ldots,t$ and $\M_i$ are their
categories of finite-dimensional right $H_i$-modules with tensor products $\otimes_i$, there is a relation between the h-depth
$d_i := d_h(R_i,H_i)$ and the h-depth of the tensor Hopf algebras $$R:= R_1 \otimes \cdots \otimes R_t \subseteq H := H_1 \otimes \cdots \otimes H_t.$$ Recall the obvious exterior tensor product
$\odot: \prod_{i=1}^t \M_i \rightarrow \M_H$ such that 
\begin{equation}
\label{eq: 2-cat}
(U_1 \odot \cdots \odot U_t) \otimes (W_1 \odot \cdots \odot W_t) \cong
(U_1 \otimes_1 W_1) \odot \cdots \odot (U_t \otimes_t W_t),
\end{equation}
an isomorphism of right $H$-modules, for $U_i , W_i \in \M_i$. There are some obvious identifications one must make
such as $H_i \subseteq H$ as a Hopf subalgebra.  
The following generalizes earlier results for depth of subgroup pairs in a direct product, e.g.\ \cite{BKK}.  

\begin{cor}
The h-depth of the tensor Hopf subalgebra pair is $d_h(R,H) = \max \{d_i : i = 1,\ldots,t \}$. 
\end{cor}
\begin{proof}
Let $V_i := H_i / R_i^+H_i$, then observe that $V_1 \odot \cdots \odot V_t \stackrel{\cong}{\longrightarrow} V$
as $H$-modules via  the $H$-module monomorphism
$$(x + R_1^+H_1) \odot \cdots \odot (z + R_t^+H_t) \mapsto x \otimes \cdots \otimes z + \sum_i H_1 \otimes \cdots \otimes R_i^+H_i \otimes \cdots \otimes H_t,$$ which is epi by counting dimensions.  The corollary follows from Theorem~\ref{theorem-main} if
we establish $d(V, \M_H) = \max \{ d(V_i, \M_i) =:  d'_i \}_{i=1}^t$.  Note from
Eq.~(\ref{eq: 2-cat}) that for each $m \in \N$, the tensor power of $V$ simplifies to $$V^{\otimes m} \cong
V_1^{\otimes_1 m} \odot \cdots \odot V_t^{\otimes_t m}.$$
Suppose $V_i^{\otimes_i (n+1)}\oplus U_i \cong q_i V_i^{\otimes_i n}$ for each $i = 1,\ldots,t$,
some complementary module $U_i$ and $q_i \in \N$, then taking the exterior tensor product of the $t$
iso-equations using the distributive law and the displayed equation obtains
$V^{\otimes (n+1)} \| q_1\cdots q_t V^{\otimes n}$.  From this it follows
that $d(V,\M_H) \leq \max \{ d'_1,\ldots,d'_t \}$.  

Conversely, the generalized permutation module $V$ restricts in $t$ ways to multiples
of $V_i \in \M_i$:  $V|_{H_i} = r_i V_i$ where $r_i = \prod_{j \neq i} (\dim V_j)$
for obvious reasons.  If $V^{\otimes (n+1)} \| t V^{\otimes n}$ for some $n \in \N$, by
restriction of this isomorphism to each category $\M_i$ one obtains $r_i^{n+1} V_i^{\otimes_i (n+1)}
\| r_i^n t V_i^{\otimes_i n}$, thus $V_i$ has depth $n$ for each $i = 1,\ldots,t$. 
Hence $d(V, \M_H) \geq \max \{ d'_1,\ldots, d'_t \}$. The corollary follows. 
\end{proof}
\section{Tensor product theorems give finite depth}
An example of a Hopf subalgebra not satisfying the hypotheses above can be constructed in characteristic $p$ from noncyclic $p$-groups and their finite group extensions.   Such examples have finite depth for another reason.  In the last two sections, we noted that when $R$ has the Chevalley  property, so that $R$-simples satisfy a tensor product theorem, then $R$ has finite depth in a Hopf algebra where the generalized permutation module $V_R$ is semisimple.  Something similar is true
for finite group algebra extensions due to a (Mackey-theoretic) tensor product theorem \cite[10.18]{CR}, which has finite closed form for permutation modules of subgroups (the Burnside ring of the group).  

Suppose a finite set of Hopf subalgebras $R_i \subseteq H$ have
induced modules $V_i = H/R_i^+H$ for $i = 1,\ldots, n$, where the  modules $V_i$ satisfy a tensor product theorem; then each depth $d(R_i,H)$ is finite, as shown below, the classical case where $H$ is a group algebra being noted. 

\begin{theorem}
Suppose $V_i \otimes V_j \cong \sum_{k =1}^n \! \oplus \ a^k_{ij} V_k$ for $n^3$ nonnegative integers $a^k_{ij}$. Then $d(R_i,H) \leq 2n+2$ for each Hopf subalgebra $R_i \subseteq H$. In particular, 
finite group algebra extensions have finite depth \cite{BDK}. 
\end{theorem}
\begin{proof}
Since each $V_i$ and its tensor powers $V_i^{\otimes n}$ are direct sums of multiples (of the form $a^k_{ii}a^r_{ki}a^s_{ri}\ldots$) of 
$\mathcal{V} := \{ V_1,\ldots, V_n \}$, the chain of $\mathcal{V}$-constituents $$\cdots \subseteq \{ V_j \in \mathcal{V} \, : \, V_j \| T_r(V_i) \} \subseteq \{ V_j \in \mathcal{V} \, : \, V_j \| T_{r+1}(V_i) \} \subseteq \cdots$$ stabilizes after at most $n$ steps $r = 1,\ldots,n$.  Thus
$T_{n+1}(V_i) \| qT_n(V_i)$ for some large enough $q$.  Apply
Theorem~\ref{theorem-main} to conclude that $d(R_i,H) \leq 2n+2$.  
 
Let $\mathcal{H}_1, \mathcal{H}_2 < \mathcal{G}$ be two subgroups of a finite group.  Consider the group algebras $k[\mathcal{H}_i] \subseteq k[\mathcal{G}]$ a Hopf subalgebra pair where the permutation module $V_i = \mathrm{Ind}^{\mathcal{G}}_{\mathcal{H}_i} k := (k|_{\mathcal{H}_i})^\mathcal{G}$ has dimension $|\mathcal{G}: \mathcal{H}_i|$ ($i=1,2$). The tensor product theorem (valid as well for a commutative ground ring) \cite[10.18]{CR}  applied to $V_i $ gives
\begin{equation}
\label{eq: tpt}
  V_1 \otimes V_2 \cong \sum_{x^{-1}y \in D}\! \oplus \, ( k|_{{}^{x}\mathcal{H}_1 \cap {}^{y}\mathcal{H}_2})^\mathcal{G}
\end{equation}
where $D$ is a set of  representatives of $(\mathcal{H}_1,\mathcal{H}_2)$-double cosets in $\mathcal{G}$ and ${}^y\mathcal{H}_i$ denotes the conjugate subgroup $y\mathcal{H}_iy^{-1}$.  Thus given a subgroup $\mathcal{H}$
and its associated permutation module $V$, let $I$ be the finite index set of  conjugate
subgroups of $\mathcal{H}$ and their intersecting subgroups, i.e.
$I = \{ {}^x\mathcal{H} \cap \cdots \cap {}^z\mathcal{H} \, : \, x,\ldots,z \in \mathcal{G} \}$, $n = |I|$ and $V_i$ the permutation modules associated with each of these, beginning with $V_1 = V$.  Then
the formula~(\ref{eq: tpt}) for $V^{\otimes 2}$ and its extension to $V_i \otimes V_j$ are indeed formulas as in the hypothesis of the corollary.  Thus, $d(V, \M_{k[\mathcal{G}]}) \leq n$, then apply Lemma~\ref{lemma-RES} and Theorem~\ref{theorem-main}. 
\end{proof}
 For example, if $\mathcal{H} \lhd \mathcal{G}$, then $ |I| = 1$; if $\G$ is a Frobenius group and $\H$ is the complementary subgroup, then $| I | = 2$. Putting together \cite[Theorem 6.9]{BKK} and Eq.~(\ref{eq: main}), one shows the depth of $V_{\C \mathcal{G}}$ is less than or equal to the minimum number of conjugates of $\mathcal{H}$ intersecting in the core of $\mathcal{H}$.

\subsection{Acknowledgements}   The author thanks Sebastian Burciu, Christian Lomp,  Alexander  Stolin, and Christopher Young for interesting discussions in Porto and Str\"omstad, as well as the University of Pennsylvania for making available its Visiting Scholar program. My thanks to the referee for corrections and pointing out algebraic modules. Research for this paper was funded by the European Regional Development Fund through the programme {\tiny COMPETE} 
and by the Portuguese Government through the FCT  under the project 
\tiny{ PE-C/MAT/UI0144/2011.nts}.

\end{document}